\documentclass[a4paper,10pt]{article}
\usepackage{amsmath}
\allowdisplaybreaks
\usepackage{stmaryrd}
\usepackage{amsfonts}
\usepackage{bbm}
\usepackage{amscd}
\usepackage{mathrsfs}
\usepackage{latexsym,amssymb,amsmath,amscd,amscd,amsthm,amsxtra}
\usepackage[dvips]{graphicx}
\usepackage[utf8]{inputenc}
\usepackage[T1]{fontenc}
\usepackage{lmodern}
\usepackage{amssymb}
\usepackage[all]{xy}
\usepackage{nicefrac,mathtools,enumitem}
\usepackage[numbers,sort&compress]{natbib}
\usepackage{microtype}
\usepackage{xcolor}
\usepackage[all]{xy}
\usepackage{listings}

\newtheorem{thm}{Theorem}[section]
\newtheorem{lem}[thm]{Lemma}
\newtheorem{cor}[thm]{Corollary}
\newtheorem{pro}[thm]{Proposition}
\newtheorem{ex}[thm]{Example}

\newtheorem{defi}[thm]{Definition}
\numberwithin{equation}{section}
\newcommand {\emptycomment}[1]{}

\newcommand{\lon }{\,\rightarrow\,}
\newcommand{\be }{\begin{equation}}
\newcommand{\ee }{\end{equation}}

\newcommand{\g}{\frkg}

\newcommand{\huaC}{{\mathcal{C}}}

\def\a{\alpha}
\def\b{\beta}
\def\f{\phi}
\def\p{\psi}
\def\r{\rho}
\def\P{\Pi}
\def\c{\cdot}

\newcommand{\frkg}{\mathfrak g}

\newcommand{\br}[1]{   [ \cdot,    \cdot  ]_\frkg   }

\newcommand{\Id}{\rm{Id}}

\newcommand{\gl}{\mathfrak {gl}}

\newcommand{\ad}{\mathrm{ad}}

\begin{document}
\title{Cohomology and linear deformation of BiHom-left-symmetric algebras}
\author{\bf A. B. Hassine, T. Chtioui, S. Mabrouk, O. Ncib}
\author{{A. B. Hassine$^{1,2}$
 \footnote{Corresponding author,  E-mail: benhassine.abdelkader@yahoo.fr }
, T. Chtioui$^{2}$
 \footnote{Corresponding author,  E-mail: chtioui.taoufik@yahoo.fr}
, S. Mabrouk$^{3}$
 \footnote{Corresponding author,  E-mail: Mabrouksami00@yahoo.fr}
, O. Ncib $^{3}$ }
 \footnote{Corresponding author,  E-mail: othmenncib@yahoo.fr}
\\
{\small 1.  Department of Mathematics, Faculty of Science and Arts at Belqarn,}\\ {\small P. O. Box 60, Sabt Al-Alaya 61985, University of Bisha, Saudi Arabia}\\
{\small 2.  Faculty of Sciences, University of Sfax, BP 1171, 3000 Sfax, Tunisia}\\
{\small 3.  University of Gafsa, Faculty of Sciences Gafsa, 2112 Gafsa, Tunisia}
}
\date{}
\maketitle
\begin{abstract}
The aim of this work is to introduce representations  of   BiHom-left-symmetric algebras. and develop its cohomology theory. As applications, we study linear deformations of BiHom-left-symmetric algebras, which are characterized by its second cohomology group with the coefficients in the adjoint representation. The notion of a Nijenhuis operator on a BiHom-left-symmetric algebra is introduced. We will prove that it  can generate trivial linear deformations of a BiHom-left-symmetric algebra.
\end{abstract}


\section*{Introduction}
Left-symmetric algebras $($called also pre-Lie algebras$)$ arise in many areas of mathematics and physics. They have already been introduced by A. Cayley in 1896, in the context of
rooted tree algebras. Then they were forgotten for a long time until Vinberg
in 1960  and Koszul in 1961who introduced them in the
context of convex homogeneous cones.

Representations and cohomology theory  of left-symmetric algebras were studied by  Askar Dzhumadil'daev in \cite{DA}.  Recently,  representations and cohomology of Hom-left-symmetric algebras were developed  in  \cite{sheng3,rephomprelie}.

In the present paper, we will study the representation and the cohomology theory of BiHom-left-symmetric algebras.  Recall that  a
 BiHom-algebra is an algebra in such a way that the identities defining the structure are
twisted by two homomorphisms $\a$ and $\b$. This class of algebras was introduced from a categorical approach in \cite{bihomass}  as an extension of the class of Hom-algebras \cite{MS2}.  Other classes of BiHom-algebras can be found in \cite{bihomalt,bihomprealt,bihomliesuper}.

The notion of a BiHom-left-symmetric algebra (or BiHom-pre-Lie algebras) was introduced in \cite{Liu&Makhlouf&Menini&Panaite}.

 There is a close relationship between BiHom-left-Lie algebras and BiHom-Lie algebras: a BiHom-left-symmetric algebra $(A,\cdot,\alpha,\b)$ gives rise to a BiHom-Lie algebra $(A,[\cdot,\cdot]_C,\alpha,\b)$ via the commutator bracket (in the BiHom sense), which is called the sub-adjacent BiHom-Lie algebra and denoted by $A^C$. On the other hand, any BiHom-Lie algebra gives rise to a BiHom-left-symmetric algebra using the Rota-Baxter operator.

Despite there are important applications of BiHom-left-symmetric algebras,   cohomologies of BiHom-left-symmetric algebras have not been well developed due to its complexity. The aim of this work  is to develop  representations and cohomology theories of this class of algebras and give some of  its applications.

In the present work, we give  the natural formula of the representation of a BiHom-left-symmetric algebra and provide some related results. Moreover, we show that there is well-defined tensor product of two representations of a BiHom-left-symmetric algebra. This is the content of Section 2.

In Section 3, we give the cohomology theory of a BiHom-left-symmetric algebra $(A,\cdot,\a,\b)$ with the coefficient in a representation $(V,L,R,\f,\p)$. The main contribution is to define the coboundary operator $\partial: C^n(A;V)\longrightarrow C^{n+1}(A;V)$.

In Section 4, we study linear deformations of a BiHom-left-symmetric algebra using the cohomology with the coefficient in the adjoint representation defined in Section 3.
We introduce the notion of a Nijenhuis operator on a BiHom-left-symmetric algebra and show that it gives rise to a trivial deformation. Finally, we study the relation between linear deformations of a BiHom-left-symmetric algebra and linear deformations of its sub-adjacent BiHom-Lie algebra.

We work over a base field $\mathbb{K}$ of characteristic zero. All algebras, linear spaces etc. will be over $\mathbb{K}$.
Throughout the paper, all the BiHom-algebras are considered regular.

\section{Preliminaries}
In this section, we recall definition, representations and some key results of BiHom-Lie algebras \cite{bihomass} and introduce the notion of an $\mathcal{O}$-operator on a BiHom-Lie algebra.

\begin{defi}
A BiHom-Lie algebra is a tuple  $(\g,[\cdot,\cdot],\a,\b)$ consisting of a linear space $\g$, a linear map $[\cdot,\cdot]:\otimes^2\g\longrightarrow \g$ and two algebra morphisms $\a,\b:\g\longrightarrow \g $, satisfying
\begin{align}
 \a \circ \b =& \b \circ \a, \nonumber\\
 [\b(x),\a(y)]=& -[\b(y),\a(x)],\nonumber\\
 \circlearrowleft_{x,y,z \in \g} [\b^2(x)&,[\b(y),\a(x)]]= 0\label{Bihom-jaco},
\end{align}
for any $x,y,z \in \g$.
\end{defi}

\begin{defi}
A morphism of BiHom-Lie algebra $f:(\g,[\cdot,\cdot]_\g,\a,\b)\longrightarrow (\mathfrak{g'},[\cdot,\cdot]_\mathfrak{g'},\a',\b')$ is a linear map $f:\g\longrightarrow \mathfrak{g'}$ such that
 \begin{eqnarray*}
 &&f[x,y]_\g=[f(x),f(y)]_\mathfrak{g'},\quad \forall~x,y\in \g,\\
&& f\circ \alpha=\a' \circ f,  ~~ f\circ \b=\b' \circ f.
  \end{eqnarray*}
  \end{defi}

 \begin{defi}\label{defi:bihom-lie representation}
 A representation of a BiHom-Lie algebra $(\g,[\cdot,\cdot],\alpha,\beta)$ on
 a vector space $V$ with respect to commuting linear maps $\phi,\psi:V\rightarrow V$ is a linear map
  $\rho:\g\longrightarrow \gl(V)$, such that for all
  $x,y\in \g$, the following equalities are satisfied
\begin{eqnarray}
\label{bihom-lie-rep-1}\rho(\alpha(x))\circ \f&=&\f\circ \rho(x),\\
\label{bihom-lie-rep-2} \rho(\b(x))\circ \p&=&\p \circ \rho(x),\\
\label{bihom-lie-rep-3}\rho([\b(x),y])\circ  \p &=&\rho(\alpha \b(x))\circ\rho(y)-\rho(\b(y))\circ\rho(\a(x)).
\end{eqnarray}
  \end{defi}
We denote such a representation by $(V,\r,\f,\p)$. For all $x\in\mathfrak{g}$, we define $\ad_{x}:\mathfrak{g}\lon \mathfrak{g}$ by
\begin{eqnarray*}
\ad_{x}(y)=[x,y],\quad\forall y \in \mathfrak{g}.
\end{eqnarray*}
Then $\ad:\g\longrightarrow\gl(\frak g)$ is a representation of the BiHom-Lie algebra $(\mathfrak{g},[\cdot,\cdot],\a,\b)$ on $\g$ with respect to $\alpha$ and $\b$,  which is called the adjoint representation.

\begin{lem}\label{lem:semidirectp}
Let $(\g,[\cdot,\cdot],\alpha,\b)$ be a BiHom-Lie algebra, $(V,\f,\p)$  a vector space with two commuting linear transformations and $\rho:
\g\rightarrow \gl(V)$ a linear
map. Then $(V,\r,\f,\p)$ is a representation of $(\g,[\cdot,\cdot],\a,\b)$ if and only if $(\g\oplus V,[\cdot,\cdot]_\rho,\alpha+\f,\b+\p)$ is a BiHom-Lie algebra, where $[\cdot,\cdot]_\rho$, $\alpha+\f$ and $\b+\p$ are defined by
\begin{eqnarray}\label{eq:sum}
[x+u,y+v]_{\rho}&=&[x,y]+\rho(x)v-\rho(\a^{-1}\b(y))\f\p^{-1}u,\\
(\alpha+\f)(x+u)&=&\alpha(x)+\f(u),\\
(\b+\p)(x+u)&=&\b(x)+\p(u),
\end{eqnarray}
for all $x,y\in \g,~u,v\in V$.
\end{lem}

This BiHom-Lie algebra is called the semidirect product of $(\g,[\cdot,\cdot],\alpha,\b)$ and $(V,\f,\p)$ and denoted by $\g\ltimes_\rho V.$

Now we introduce the definition  of an $\mathcal{O}$-operator on a BiHom-Lie algebra associated to a given representation, which generalize the Rota-Baxter operator of weight $0$ introduced in \cite{Liu&Makhlouf&Menini&Panaite} and defined as a linear operator $R$ on a BiHom-Lie algebra  $(\g,[\cdot,\cdot],\alpha,\beta)$ such that $R\circ \a=\a \circ R$, $R \circ \b=\b \circ R$ and
$$[R(x),R(y)]=R([R(x),y]+[x,R(y)]),\ \forall x,y \in \g.$$
\begin{defi}
Let $(\g,[\cdot,\cdot],\alpha,\beta)$ be a BiHom-Lie algebra and $(V,\rho,\f,\p)$ be a representation.  A linear map $T: V \to A$ is called an $\mathcal{O}$-\textbf{operator}
associated to $\rho$, if it satisfies
\begin{align}\label{O-operator}
    [T(u),T(v)]=T(\rho(T(u))v-\rho(T(\f^{-1}\p(v)))\f\p^{-1}(u)),\ \ \forall u,v \in V.
\end{align}
\end{defi}

Let $(A,\cdot)$ be a left-symmetric algebra. The commutator $[x,y]_A=x\cdot y-y\cdot x$ defines a Lie algebra structure on $A$, which is called the sub-adjacent Lie algebra of $A$ and denoted by $\g(A)$. Furthermore, $L:A\longrightarrow \gl(A)$ with $L(x) y=x\cdot y$ gives a representation of the Hom-Lie algebra $\g(A)$ on $A$.

\begin{defi}
Let $(A,\cdot)$ be a left-symmetric algebra and $V$ a vector space.  A representation of $A$ on $V$ consists of a pair $(L,R)$, where $L: A\longrightarrow \gl(V)$ is a representation of the Lie algebra $\g(A)$ on $A$ and $R:A\longrightarrow \gl(V)$ is a linear map satisfying
\begin{eqnarray}
 \L(x)\circ R(y)-R(y)\circ L(x)=R(x\cdot y)-R(y)\circ L(x),\quad \forall~x,y\in A.
\end{eqnarray}
  \end{defi}
Usually, we denote a representation by $(V,L,R)$. Define $\ell,r:A\longrightarrow\gl(A)$ by $\ell(x)y=x\cdot y=r(y)x$. Then $(A,\ell,r)$ is a representation of $(A,\cdot)$ called the adjoint representation.

\section{Representations of BiHom-left-symmetric algebras}

Let $(A,\cdot,\alpha,\b)$ be a BiHom-left-symmetric algebra. We always assume that it is regular, i.e. $\alpha, \b$ are invertible. The commutator $[x,y]_C=x\cdot y-\a^{-1}\b(y)\cdot\a\b^{-1} ( x)$ gives a BiHom-Lie algebra $(A,[\cdot,\cdot]_C,\alpha,\b)$, which is denoted by $A^C$ and called  the sub-adjacent BiHom-Lie algebra of $(A,\cdot,\alpha,\b)$. For more details see \cite{Liu&Makhlouf&Menini&Panaite}.

\begin{defi}
A BiHom-left-symmetric algebra is a tuple
$(A,\cdot,\a,\b)$ consisting of a vector space $A$ where $\cdot: A\otimes A \to A$ is a linear map and  $\a,\b: A \to A$ are two linear maps  such that for any $x,y,z \in A$
\begin{align}\label{id-Bihom-pre}
 \a \circ \b= &\b \circ \a,\\
 \a(x \cdot y)=\a(x)\cdot \a(y),  & \b(x \cdot y)=\b(x)\cdot \b(y),\\
(\b(x)\cdot \a(y))\cdot \a\b(z)-\a\b(x)\cdot (\a(y)\cdot z)=&(\b(y)\cdot \a(x))\cdot \a\b(z)-\a\b(y)\cdot (\a(x)\cdot z).
\end{align}
\end{defi}

\begin{defi}
 A morphism from a BiHom-left-symmetric algebra $(A,\cdot,\alpha,\b)$ to a BiHom-left-symmetric algebra $(A',\cdot',\alpha',\b')$ is a linear map $f:A\longrightarrow A'$ such that for all
  $x,y\in A$, the following equalities are satisfied
\begin{eqnarray}
\label{bihomo-1}f(x\cdot y)&=&f(x)\cdot' f(y),\\
\label{bihomo-2}f\circ \alpha&=&\alpha'\circ f,\\
\label{bihomo-3}f\circ \b&=&\b'\circ f.
\end{eqnarray}
\end{defi}

\begin{defi}\label{defi:hom-pre representation}
 A representation of a BiHom-left-symmetric algebra $(A,\cdot,\alpha,\b)$ on a vector space $V$ with respect to two commuting linear maps $\phi,\psi:V\rightarrow V$ consists of a pair $(L,R)$, where $L,R:A\longrightarrow \gl(V)$ are two linear maps satisfying, for all $x,y\in A$,
\begin{align}
&\f L(x)=L(\a(x))\f,~~  \p L(x)=L(\b(x))\p, ~~ \f R(x)=R(\a(x))\f,~~  \p R(x)=R(\b(x))\p,\label{rep1}\\
&L(\b(x)\cdot \a(y))\p-L(\a\b(x))L(\a(y))=L(\b(y)\cdot\a(x))\p-L(\a\b(y))L(\a(x)),\label{rep2}\\
&R(\b(x))L(\b(y))\f-L(\a\b(y))R(x)\f=R(\b(x))R(\a(y))\p-R(\a(y)\cdot x)\f\p.\label{rep3}
\end{align}
\end{defi}
We denote a representation of a BiHom-left-symmetric algebra $(A,\cdot,\alpha,\b)$ by $(V,L,R,\f,\p)$. We define the linear operation $\cdot_\ltimes:\otimes^2(A\oplus V)\lon(A\oplus V)$ by
\begin{equation}
(x+u)\cdot_\ltimes (y+v):=x\cdot y+L(x)(v)+R(y)(u),\quad \forall x,y \in A,  u,v\in V.
\end{equation}\label{eq:12.8}
and the linear maps $\alpha+\f,\b+\p:A\oplus V\longrightarrow A\oplus V$ by
\begin{equation}
(\alpha+\beta)(x+u):=\alpha(x)+\f(u),\quad (\b+\p)(x+u):=\b(x)+\p(u),\quad  \forall x\in A,  u\in V.
\end{equation}

\begin{pro}\label{direct-product}
 With the above notations, $(A\oplus V,\cdot_\ltimes,\alpha+\f,\b+\p)$ is a BiHom-left-symmetric algebra if and only if $(V,L,R,\f,\p)$ is a representation of the  BiHom-left-symmetric algebra $(A,\cdot,\alpha,\b)$.  \\
It is denoted by $A\ltimes_{(L,R)}V$ and called  the  semidirect product of  $(A,\cdot,\alpha,\b)$ and $(V,\f,\p)$.
\end{pro}
\begin{proof}
Let $x,y,z \in A$ and $u,v,w, \in V$. Then
\begin{align*}
&((\b(x)+\p(u)) \cdot_\ltimes(\a(y)+\f(v)))\cdot_\ltimes(\b(z)+\p(w))- (\a\b(x)+\f\p(u))\cdot_\ltimes ((\a(y)+\f(v))\cdot_\ltimes(z+w))\\
=&(\b(x)\cdot \a(y)+L(\b(x))\f(v)+R(\a(y))\p(u))\cdot_\ltimes(\b(z)+\p(w))\\
-&(\a\b(x)+\f\p(u))\cdot_\ltimes (\a(y)\cdot z +L(\a(y))w+R(z)\f(v))\\
=&(\b(x)\cdot \a(y))\cdot \b(z)+L(\b(x)\cdot \a(y))\p(w)+R(\b(z))L(\b(x))\f(v)-R(\b(z))R(\a(y))\p(u)\\
-& \a\b(x)\cdot (\a(y)\cdot z)-L(\a\b(x))L(\a(y))w-L(\a\b(x))R(z)\f(v)-R(\a(y)\cdot z)\f\p(u).
\end{align*}
On the other hand,
\begin{align*}
&((\b(y)+\p(v)) \cdot_\ltimes(\a(x)+\f(u)))\cdot_\ltimes(\b(z)+\p(w))- (\a\b(y)+\f\p(v))\cdot_\ltimes ((\a(x)+\f(u))\cdot_\ltimes(z+w))\\
=&(\b(y)\cdot \a(x)+L(\b(y))\f(u)+R(\a(x))\p(v))\cdot_\ltimes(\b(z)+\p(w))\\
-&(\a\b(y)+\f\p(v))\cdot_\ltimes (\a(x)\cdot z +L(\a(x))w+R(z)\f(u))\\
=&(\b(y)\cdot \a(x))\cdot \b(z)+L(\b(y)\cdot \a(x))\p(w)+R(\b(z))L(\b(y))\f(u)-R(\b(z))R(\a(x))\p(v)\\
-& \a\b(y)\cdot (\a(x)\cdot z)-L(\a\b(y))L(\a(x))w-L(\a\b(y))R(z)\f(u)-R(\a(x)\cdot z)\f\p(v).
\end{align*}
Using the fact that $(A,\cdot,\a,\b)$ is a BiHom-left-symmetric algebra and identities \eqref{rep2}-\eqref{rep3}, we deduce that $(A\oplus V,\cdot_\ltimes,\alpha+\f,\b+\p)$ is a BiHom-left-symmetric algebra.
\end{proof}

\begin{ex}
Let $(A,\cdot,\a, \b)$ be a BiHom-left-symmetric algebra. Define $\ell,r:A\longrightarrow \gl(A)$ by
$$\ell_x(y)=x\cdot y,\quad r_x(y)=y\cdot x, \quad \forall x,y \in A.$$
Then $(A,\ell,r,\a,\b)$ is a representation of $(A,\cdot,\alpha,\b)$, which is called the  adjoint representation.
\end{ex}

\begin{ex}
Let $(A,\cdot,\alpha,\b)$ be a BiHom-left-symmetric algebra. Then $(\mathbb{R},0,0,1,1)$ is a representation of $(A,\cdot,\alpha,\b)$, which is called the trivial representation.

\end{ex}
\begin{pro}\label{important-rep}
Let $(V,L,R,\f,\p)$ be a representation of a BiHom-left-symmetric algebra $(A,\cdot,\alpha,\b)$. Then we have the following assertions
\begin{itemize}
  \item [(i)]  $(V,L,\f,\p)$ is a representation of the sub-adjacent BiHom-Lie algebra $A^C$.
  \item [(ii)] $(V,\r,\f,\p)$ is a  representation of the sub-adjacent BiHom-Lie algebra $A^C$,
   where for any $x \in A$,
\begin{align}\label{representation induite}
& \r(x)=L(x)-R(\a\b^{-1}(x))\circ \f^{-1}\p.
\end{align}
\end{itemize}
\end{pro}

\begin{proof}
(i) Using Eq. \eqref{rep2} we can easily check that $L$ is a representation of $A^C$ on $V$ with respect to $\f,\ \p$.

(ii) By Proposition \ref{direct-product},  the semidirect product $A\ltimes_{(L,R)}V$ is a BiHom-left-symmetric algebra . Consider its subadjacent BiHom-Lie algebra structure $(A\oplus V,[\cdot,\cdot]_{C},\alpha+\f,\b+\beta)$, we have
\begin{align*}
&[x+u,y+v]_{C}=(x+u)\cdot_\ltimes(y+v)-(\a+\f)^{-1}(\b+\p)(y+v)\cdot_\ltimes (\a+\f)(\b+\p)^{-1}(x+u)\\
=&(x+u)\cdot_\ltimes(y+v)-(\a^{-1}\b(y)+\f^{-1}\p(v))\cdot_\ltimes(\a\b^{-1}(x)+\f\p^{-1}(u))\\
=&x\cdot y+L(x)v+R(y)u-\a^{-1}\b(y)\cdot \a\b^{-1}(x)-L(\a^{-1}\b(y))\f\p^{-1}(u)-R(\a\b^{-1}(x))\f^{-1}\p(v)\\
=&[x,y]_C+(L(x)-R(\a\b^{-1}(x))\f^{-1}\p)(v)-(L(\a^{-1}\b(y))\f\p^{-1}-R(y))(u).
\end{align*}
According to Lemma \ref{lem:semidirectp}, we deduce that $(V,\rho,\f,\p)$ is a representation of the BiHom-Lie algebra $A^C$.
\end{proof}

\

Let $(V,L,R)$ be a representation of a left-symmetric algebra $(A,\cdot)$.  Given four linear maps $\a,\b: A \to A$ and $\f,\p: V \to V$ satisfying
\begin{align*}
 \a\b= \b\a,&\  \f\p=\p\f,\\
 \f L(x)=L(\a(x))\f,&\ ~~~~~\ \ \  \p L(x)=L(\b(x))\p, \\
   \f R(x)=R(\a(x))\f,&\  \ \ \   \p R(x)=R(\b(x))\p.
\end{align*}
Then, define $\mathfrak{L}, \mathfrak{R}: A \to \gl(V)$ and $\cdot_{\a,\b}: A^{\otimes 2}\to A$ respectively by $\mathfrak{L}(x)=L(\a(x))\p$,
$\mathfrak{R}(x)=R(\b(x))\f$ and $x\cdot_{\a,\b}y=\a(x)\cdot \b(y)$.
We have the following conclusion
\begin{pro}
Under the above notations, $(V,\mathfrak{L},\mathfrak{R},\f,\p)$ is a representation of the BiHom-left-symmetric algebra $(A,\cdot_{\a,\b},\a,\b)$.
\end{pro}

\begin{proof}
Let $x,y \in A$ and $v \in V$. It is easy to check that
\begin{align*}
 \f \mathfrak{L}(x)=\mathfrak{L}(\a(x))\f,&\  \ \ \   \p \mathfrak{L}(x)=\mathfrak{L}(\b(x))\p, \\
   \f \mathfrak{R}(x)=\mathfrak{R}(\a(x))\f,&\  \ \ \   \p \mathcal{R}(x)=\mathfrak{R}(\b(x))\p.
\end{align*}
To prove \eqref{rep2}, we compute as follows
\begin{align*}
&\mathfrak{L}(\b(x) \cdot_{\a,\b}\a(y))\p(v)-\mathfrak{L}(\a\b(x))\mathfrak{L}(\a(y))v\\
=& \mathfrak{L}(\a\b(x)\cdot \a\b(y))\p(v)-\mathfrak{L}(\a\b(x))L(\a^2(y))\p(v)\\
=& L(\a^2\b(x)\cdot \a^2\b(y))\p^2(v)-L(\a^2\b(x))L(\a^2\b(y))\p^2(v)\\
=&L(\a^2\b(y)\cdot \a^2\b(x))\p^2(v)-L(\a^2\b(y))L(\a^2\b(x))\p^2(v)\\
=&\mathfrak{L}(\b(y) \cdot_{\a,\b}\a(x))\p(v)-\mathfrak{L}(\a\b(y))\mathfrak{L}(\a(x))v.
\end{align*}
Similarly, one can get identity \eqref{rep3}.
\end{proof}

\begin{thm} \label{lie to left-symmetric}
Let $T:V \to \g$ be an $\mathcal{O}$-operator on a BiHom-Lie algebra $(\g,[\cdot,\cdot],\a,\b)$ associated to a representation $(V,\rho,\f,\p)$.  Then there exists a BiHom-left-symmetric algebra structure on $V$ defined by
\begin{equation}\label{lie==>left-symmetric}
    u \ast v= \rho(T(u))v,\ \forall u,v \in V.
\end{equation}
Therefore $V$ is a BiHom-Lie algebra as the sub-adjacent
BiHom-Lie algebra of this BiHom-left-symmetric algebra and $T$ is a homomorphism of BiHom-Lie algebras. Furthermore,
$T(V)=\{T(v)| v \in V\}\subseteq \g$ is a BiHom-Lie subalgebra of $\g$ and there is an induced BiHom-left-symmetric algebra structure on $T(V )$ given by
\begin{equation}
    T(u) \circ T(v)=T(u \ast v)=T(\rho(T(u))v),\ \forall u,v \in V.
\end{equation}
Moreover, $T$ is a homomorphism of BiHom-left-symmetric algebras.
\end{thm}

\begin{proof}
Let $u, v, w \in V$. Then we have
\begin{align*}
&(\p(u)\ast \f(v))\ast \p(w)-\f \p(u)\ast (\f(v) \ast w) \\
=& (\rho(T(\p(u))) \f(v)) \ast \p(w)-\f\p(u) \ast (\rho(T(\f(v))) w) \\
=& \rho ( T(\rho(T(\p(u))) \f(v)) ) \p(w)-\rho( T(\f\p(u) ) )\rho(T(\f(v))) w \\
=& \rho[T\p(u),T\f(v)]\p(w)+\rho(T(\rho(T(\p(v)))\f(u))\p(w) \\
-&\rho[T\p(u),T\f(v)]\p(w)-\rho(T\f\p(v))\rho(T\f(u))w  \\
=& \rho(T(\rho(T(\p(v)))\f(u))\p(w)-\rho(T\f\p(v))\rho(T\f(u))w \\
=& (\p(v)\ast \f(u))\ast \p(w)-\f \p(v)\ast (\f(u) \ast w) .
\end{align*}
Hence $(V,\ast,\f,\p)$ is a BiHom-left-symmetric algebra.
On the other hand,
\begin{align*}
 T([u,v]_{V^C})=& T(u \ast v-\f^{-1}\p(v)\ast \f\p^{-1}(u) ) \\
     =& T(\rho(T(u))v-\rho(T(\f^{-1}\p(v)))\f\p^{-1}(u)) \\
     =& [T(u),T(v)],
\end{align*}
which implies that $T$ is a homomorphism of BiHom-Lie algebras. The other conclusions follow immediately.
\end{proof}

A direct consequence of Theorem \ref{lie to left-symmetric} is the following construction of a
BiHom-left-symmetric algebra from a Rota-Baxter operator (of weight zero) of a BiHom-Lie
algebra.

\begin{cor}\cite{Liu&Makhlouf&Menini&Panaite}
Let $(\g,[\cdot,\cdot],\a,\b)$ be a BiHom-Lie algebra and $R: \g \rightarrow \g$ be a Rota-Baxter operator of weight zero. Then there is a BiHom-left-symmetric algebra structure on $\g$ given by
\begin{eqnarray}
 x \ast y=[R(x),y]\ \forall x,y\in \g.
\end{eqnarray}
\end{cor}

\begin{pro}
Let $(\g,[\cdot,\cdot],\a,\b)$ be a BiHom-Lie algebra. Then there exists a compatible BiHom-left-symmetric algebra structure on $\g$ if and only if there exists an invertible $\mathcal{O}$-operator of $(\g,[\cdot,\cdot],\a,\b)$.
\end{pro}

\begin{proof}
If there exists an invertible $\mathcal{O}$-operator of $(\g,[\cdot,\cdot],\a,\b)$ associated to a representation $(V,\rho,\f,\p)$, then by Theorem \ref{lie to left-symmetric}, there is a BiHom-left-symmetric algebra structure on $\g$ given by
$$x \circ y= T(\rho(x)T^{-1}(y)),\ \forall x,y \in \g.$$
Conversely, let $(\g,\cdot,\a,\b)$ be a BiHom-left-symmetric  algebra and $(\g,[\cdot,\cdot],\a,\b)$ be the associated BiHom-Lie algebra. Then the identity map $id: \g \to \g$ is an $\mathcal{O}$-operator of $(\g,[\cdot,\cdot],\a,\b)$
associated to $(\g,\ad,\a,\b)$.
\end{proof}


At the end of this section, we show that the tensor product of two representations of a BiHom-left-symmetric algebra is still a representation.

\begin{pro}

Let $(A,\cdot,\a,\b)$ be a BiHom-left-symmetric algebra, $(V,L_V,R_V,\f_V,\p_V)$ and $(W,L_W,R_W,\f_W,\p_W)$ two representations of $A$. Then $(V\otimes W,L_V\otimes \p_W+\p_V\otimes (L_W-R_W(\a\b^{-1}(\cdot))\f_W^{-1}\p_W),R_V\otimes \f_W,\f_V\otimes \f_W,\p_V\otimes \p_W)$ is a representation of $(A,\cdot,\a,\b)$.
\end{pro}
\begin{proof}
Since $(V,L_V,R_V,\f_V,\p_V)$ and $(W,L_W,R_W,\f_W,\p_W)$ are representations of a BiHom-left-symmetric algebra $(A,\cdot,\alpha,\b)$, for all $x\in A$, we have
\begin{eqnarray*}&&\Big(L_V\otimes \p_W+\p_V\otimes (L_W-R_W(\a\b^{-1}(\cdot))\f_W^{-1}\p_W)\Big)(\alpha(x))\circ (\f_V\otimes \f_W)\\
&&-(\f_V\otimes \f_W)\circ \Big(L_V\otimes \p_W+\p_V\otimes (L_W-R_W(\a\b^{-1}(\cdot))\f_W^{-1}\p_W)\Big)(x)\\
&=&\big(L_V(\alpha(x))\circ \p_V \big)\otimes(\p_W\circ \f_W)+(\p_V\circ \f_V)\otimes \big(L_W(\alpha(x))\circ \f_W\big)\\
&&-(\p_V\circ \f_V)\otimes \big(P_W(\alpha^2\beta^{-1}(x))\circ \p_W\big)-\big(\f_V\circ L_V(x)\big)\otimes(\f_W\circ \p_W)\\
&&-(\f_V\circ \p_V)\otimes \big(\f_W \circ L_W(x)\big)+(\f_V\circ \p_V)\otimes \big(\f_W\circ R_W(\alpha\beta^{-1}(x))\circ\f^{-1} \p_W\big)=0,
\end{eqnarray*}
which implies that
\begin{align}\label{tensor-rep-1}
&\Big(L_V\otimes \p_W+\p_V\otimes (L_W-R_W(\a\b^{-1}(\cdot))\f_W^{-1}\p_W)\Big)(\alpha(x))\circ (\f_V\otimes \f_W)\\ \nonumber=&(\f_V\otimes \f_W)\circ \Big(L_V\otimes \p_W+\p_V\otimes (L_W-R_W(\a\b^{-1}(\cdot))\f_W^{-1}\p_W)\Big)(x),
\end{align}
and
\begin{eqnarray*}&&\Big(R_V\otimes\f_W\Big)(\alpha(x))\circ (\f_V\otimes \f_W)-(\f_V\otimes \f_W)\circ \Big(R_V\otimes\f_W\Big)(x)\\
&=&(R_V(\a(x))\circ\f_V\otimes\f_W^2)-(\f_V\circ R_V(x)\otimes \f_W^2)=0,
\end{eqnarray*}
which implies that
\begin{eqnarray}\label{tensor-rep-2}\Big(R_V\otimes\f_W\Big)(\alpha(x))\circ (\f_V\otimes \f_W)=(\f_V\otimes \f_W)\circ \Big(R_V\otimes\f_W\Big)(x).
\end{eqnarray}
Similarly, we have
\begin{align}\label{tensor-rep-3}
&\Big(L_V\otimes \p_W+\p_V\otimes (L_W-R_W(\a\b^{-1}(\cdot))\f_W^{-1}\p_W)\Big)(\b(x))\circ (\p_V\otimes \p_W)\\ \nonumber=&(\p_V\otimes \p_W)\circ \Big(L_V\otimes \p_W+\p_V\otimes (L_W-R_W(\a\b^{-1}(\cdot))\f_W^{-1}\p_W)\Big)(x),
\end{align}
and
\begin{eqnarray}\label{tensor-rep-4}
\Big(R_V\otimes\f_W\Big)(\b(x))\circ (\p_V\otimes \p_W)=(\p_V\otimes \p_W)\circ \Big(R_V\otimes\f_W\Big)(x).
\end{eqnarray}
For all $x,y\in A$, by straightforward computations, we have
{\small\begin{align*}
&\Big(L_V\otimes \p_W+\p_V\otimes (L_W-R_W(\a\b^{-1}(\cdot))\f_W^{-1}\p_W)\Big)([\b(x),\a(y)])\circ (\p_V\otimes \p_W)\\
-&\Big(L_V\otimes \p_W+\p_V\otimes (L_W-R_W(\a\b^{-1}(\cdot))\f_W^{-1}\p_W)\Big)(\alpha\b(x))\circ \Big(L_V\otimes \p_W+\p_V\otimes (L_W-R_W(\a\b^{-1}(\cdot))\f_W^{-1}\p_W)\Big)(\a(y))\\
+&\Big(L_V\otimes \p_W+\p_V\otimes (L_W-R_W(\a\b^{-1}(\cdot))\f_W^{-1}\p_W)\Big)(\alpha\b(y))\circ \Big(L_V\otimes \p_W+\p_V\otimes (L_W-R_W(\a\b^{-1}(\cdot))\f_W^{-1}\p_W)\Big)(\a(x))\\
=&\big(L_V([\b(x),\a(y)])\circ \p_V \big)\otimes(\p_W\circ \p_W)+(\p_V\circ \p_V)\otimes \big(L_W([\b(x),\a(y)])\circ \p_W\big)\\
-&(\p_V\circ \p_V)\otimes \big(R_W([\a(x),\a^2\b^{-1}(y)])\circ \f_W^{-1}\circ\p_W^2\big)-\big(L_V(\alpha\b(x))\circ L_V(\a(y))\big)\otimes (\p_W\circ \p_W)\\
-&\big(L_V(\alpha\b(x))\circ \p_V\big)\otimes \big(\p_W\circ R_W(\a(y))\big)+\big(L_V(\alpha\b(x))\circ \p_V \big)\otimes \big(\p_W\circ R_W(\a^2\b^{-1}(y))\big)\f_W^{-1}\p_W^2\\
-&\big(\p_V\circ L_V(\a(y))\big)\otimes \big(L_W(\alpha\b(x))\circ \p_W\big)-(\p_V\circ \p_V)\otimes \big(L_W(\alpha\b(x))\circ L_W(\a(y))\big)\\
+&(\p_V\circ \p_V)\otimes \big(L_W(\alpha\b(x))\circ R_W(\a^2\b^{-1}(y))\big)\f^{-1}\p+\big(\p_V\circ L_V(\a(y))\big)\otimes \big(R_W(\alpha^2(x))\circ \f_W^{-1}\p_W^2\big)\\
+&(\p_V\circ \p_V)\otimes \big(R_W(\alpha^2(x))\circ \f_W^{-1}\p_W\circ L_W(\a(y))\big)-(\p_V\circ \p_V)\otimes \big(R_W(\alpha^2(x))\circ \f_W^{-1}\p_W\circ R_W(\a^2\b^{-1}(y)\big)\f_W^{-1}\p_W\\
+&\big(L_V(\alpha\b(y))\circ L_V(\a(x))\big)\otimes (\p_W\circ \p_W)-\big(L_V(\alpha\b(y))\circ \p_V\big)\otimes \big(\p_W\circ R_W(\a(x))\big)\\
-&\big(L_V(\alpha\b(y))\circ \p_V \big)\otimes \big(\p_W\circ R_W(\a^2\b^{-1}(x))\big)\f_W^{-1}\p_W+\big(\p_V\circ L_V(\a(x))\big)\otimes \big(L_W(\alpha\b(y))\circ \p_W\big)\\
+&(\p_V\circ \p_V)\otimes \big(L_W(\alpha\b(y))\circ L_W(\a(x))\big)-(\p_V\circ \p_V)\otimes \big(L_W(\alpha\b(y))\circ R_W(\a^2\b^{-1}(x))\big)\f^{-1}\p^2\\
+&\big(\p_V\circ L_V(\a(x))\big)\otimes \big(r_W(\alpha^2(y))\circ \f_W^{-1}\p_W^2\big)-(\p_V\circ \p_V)\otimes \big(R_W(\alpha^2(y))\circ \f_W^{-1}\p_W\circ L_W(\a(x))\big)\\
-&(\p_V\circ \p_V)\otimes \big(R_W(\alpha^2(y))\circ \f_W^{-1}\p_W\circ R_W(\a^2\b^{-1}(x)\big)\f_W^{-1}\p_W=0,
\end{align*}}

which implies that
{\small\begin{align}\label{tensor-rep-5}
&\Big(L_V\otimes \p_W+\p_V\otimes (L_W-R_W(\a\b^{-1}(\cdot))\f_W^{-1}\p_W)\Big)([\b(x),\a(y)])\circ (\p_V\otimes \p_W)\\
\nonumber=&\Big(L_V\otimes \p_W+\p_V\otimes (L_W-R_W(\a\b^{-1}(\cdot))\f_W^{-1}\p_W)\Big)(\alpha\b(x))\circ \Big(L_V\otimes \p_W+\p_V\otimes (L_W-R_W(\a\b^{-1}(\cdot))\f_W^{-1}\p_W)\Big)(\a(y))\\
\nonumber-&\Big(L_V\otimes \p_W+\p_V\otimes (L_W-R_W(\a\b^{-1}(\cdot))\f_W^{-1}\p_W)\Big)(\alpha\b(y))\circ \Big(L_V\otimes \p_W+\p_V\otimes (L_W-R_W(\a\b^{-1}(\cdot))\f_W^{-1}\p_W)\Big)(\a(x)).
\end{align}}
Similarly, we have
\begin{eqnarray}\label{tensor-rep-6}
\nonumber &&(\f_V\otimes \f_W)(\b(y))\circ \Big(L_V\otimes \p_W+\p_V\otimes (L_W-R_W(\a\b^{-1}(\cdot))\f_W^{-1}\p_W\Big)(\b(x))\\
\nonumber &&-\Big(L_V\otimes \p_W+\p_V\otimes (L_W-R_W(\a\b^{-1}(\cdot))\f_W^{-1}\p_W\Big)(\alpha\b(x))\circ (\f_V\otimes \f_W)(y)\\
&=&(\f_V\otimes \f_W)(\b(y))\circ (\f_V\otimes \f_W)(\a(x))-(\f_V\otimes \f_W)(\a(x)\cdot y)\circ (\beta_V\otimes \beta_W).
\end{eqnarray}
By \eqref{tensor-rep-1}, \eqref{tensor-rep-2}, \eqref{tensor-rep-3}, \eqref{tensor-rep-4}, \eqref{tensor-rep-5} and  \eqref{tensor-rep-6}, , we deduce that $(V\otimes W,L_V\otimes \p_W+\p_V\otimes (L_W-R_W(\a\b^{-1}(\cdot))\f_W^{-1}\p_W),R_V\otimes \f_W,\f_V\otimes \f_W,\p_V\otimes \p_W)$ is a representation of $(A,\cdot,\alpha,\b)$.
\end{proof}
\section{Cohomologies of BiHom-left-symmetric algebras}

In this section, we establish the cohomology theory of BiHom-left-symmetric algebras. See \cite{DA} for more details for the cohomology theory of left-symmetric algebras and \cite{rephomprelie} for the cohomology theory of Hom-left-symmetric algebras.

Let $(V,L,R,\phi,\psi)$ be a representation of a BiHom-left-symmetric algebra $(A,\cdot,\alpha,\b)$.
An $n$-cochains  is a linear map $f:\wedge^{n-1} A\otimes A\to V$ such that it is compatible with $\a$, $\b$ and $\f$, $\p$ in the sense that $\f  f = f \alpha^{\otimes n},
\p  f = f  \b^{\otimes n}$.
The set of $n$-cochains is denoted by $ C^{n}(A;V),\ \  \forall n\geq 0$.

\

For all $f \in C^n(A;V),~ x_1,\dots,x_{n+1} \in A$, define the operator
$\partial^n:C^n(A;V)\longrightarrow C^{n+1}(A;V)$ by
\begin{eqnarray}\label{eq:12.1}
 &&(\partial^n f)(x_1,\dots,x_{n+1})\\
 \nonumber&=&\sum_{i=1}^n(-1)^{i+1}L(\alpha^{n-1}\b^{n-1}(x_i))f(\alpha(x_1),\dots,\widehat{\alpha(x_i)},\dots,\alpha(x_{n}),x_{n+1})\\
\nonumber &&+\sum_{i=1}^n(-1)^{i+1}R(\b^{n-1}(x_{n+1}))f(\b(x_1),\dots,\widehat{\b(x_i)},\dots,\b(x_n),\alpha^{n-1}(x_i))\\
\nonumber &&-\sum_{i=1}^n(-1)^{i+1}  f(\alpha\b(x_1),\dots,\widehat{\alpha\b(x_i)}\dots,\alpha\b(x_n),\alpha^{n-1}(x_i)\cdot x_{n+1})\\
\nonumber&&+\sum_{1\leq i<j\leq n}(-1)^{i+j} f([\b(x_i),\alpha(x_j)]_C,\alpha\b(x_1),\dots,\widehat{\alpha\b(x_i)},\dots,\widehat{\alpha\b(x_j)},\dots,\alpha\b(x_{n}),\b(x_{n+1})).
\end{eqnarray}

\begin{lem}
With the above notations, for any $f\in C^n(A;V)$, we have
\begin{eqnarray*}
&& \f (\partial^n f)=(\partial^n f) \a^{\otimes n+1}, \\
&&\p (\partial^n f)=(\partial^n f) \b^{\otimes n+1}.
\end{eqnarray*}
Thus we obtain a well-defined map
$$\partial^n:C^n(A;V)\longrightarrow C^{n+1}(A;V).$$
\end{lem}
\begin{proof}
Let $f \in C^n(A;V),~ x_1,\dots,x_{n+1} \in A$, we have
\begin{eqnarray*}
 &&(\partial^n f)(\a(x_1),\dots,\a(x_{n+1}))\\
 \nonumber&=&\sum_{i=1}^n(-1)^{i+1}L(\alpha^{n}\b^{n-1}(x_i))f(\alpha^2(x_1),\dots,\widehat{\alpha^2(x_i)},\dots,\alpha^2(x_{n}),\a(x_{n+1}))\\
\nonumber &&+\sum_{i=1}^n(-1)^{i+1}R(\a\b^{n-1}(x_{n+1}))f(\a\b(x_1),\dots,\widehat{\a\b(x_i)},\dots,\a\b(x_n),\alpha^n(x_i))\\
\nonumber &&-\sum_{i=1}^n(-1)^{i+1}  f(\alpha^2\b(x_1),\dots,\widehat{\alpha^2\b(x_i)}\dots,\alpha^2\b(x_n),\alpha^n(x_i)\cdot x_{n+1}))\\
\nonumber&&+\sum_{1\leq i<j\leq n}(-1)^{i+j} f([\a\b(x_i),\alpha^2(x_j)]_C,\alpha^2\b(x_1),\dots,\widehat{\alpha^2\b(x_i)},\dots,\widehat{\alpha^2\b(x_j)},\dots,\alpha^2\b(x_{n}),\a\b(x_{n+1}))\\
\nonumber&=&\sum_{i=1}^n(-1)^{i+1}\f L(\alpha^{n-1}\b^{n-1}(x_i))f(\alpha(x_1),\dots,\widehat{\alpha(x_i)},\dots,\alpha(x_{n}),x_{n+1})\\
\nonumber &&+\sum_{i=1}^n(-1)^{i+1}\f R(\b^{n-1}(x_{n+1}))f(\b(x_1),\dots,\widehat{\b(x_i)},\dots,\b(x_n),\alpha^n(x_i))\\
\nonumber &&-\sum_{i=1}^n(-1)^{i+1} \f f(\alpha\b(x_1),\dots,\widehat{\alpha\b(x_i)}\dots,\alpha\b(x_n),\alpha^n(x_i)\cdot x_{n+1}))\\
\nonumber&&+\sum_{1\leq i<j\leq n}(-1)^{i+j} \f f([\b(x_i),\alpha(x_j)]_C,\alpha\b(x_1),\dots,\widehat{\alpha\b(x_i)},\dots,\widehat{\alpha\b(x_j)},\dots,\alpha\b(x_{n}),\b(x_{n+1}))\\
\nonumber&=&\f (\partial^n f)(x_1,\dots,x_{n+1}).
\end{eqnarray*}
Similarly, we obtain $(\partial^n f)(\b(x_1),\dots,\b(x_{n+1}))=\p (\partial^n f)(x_1,\dots,x_{n+1}).$
\end{proof}
\begin{thm}\label{thm:operator}
The operator $\partial^n$ defined as above satisfies $\partial^{n+1}\circ\partial^{n}=0$.
\end{thm}
\begin{proof}
 Let $f \in C^n(A;V),~ x_1,\dots,x_{n+1},x_{n+2} \in A$, we have
\begin{eqnarray}
 \nonumber&&\partial^{n+1}\circ\partial^n (f)(x_1,\dots,x_{n+2})\\
 \nonumber&=&\sum_{i=1}^{n+1}(-1)^{i+1}L(\alpha^{n}\b^{n}(x_i))\partial^nf(\alpha(x_1),\dots,\widehat{\alpha(x_i)},\dots,\alpha(x_{n+1},x_{n+2}))\\
\nonumber &&+\sum_{i=1}^{n+1}(-1)^{i+1}R(\b^n(x_{n+2}))\partial^nf(\b(x_1),\dots,\widehat{\b(x_i)},\dots,\b(x_{n+1}),\alpha^n(x_i))\\
\nonumber &&-\sum_{i=1}^{n+1}(-1)^{i+1}  \partial^nf(\alpha\b(x_1),\dots,\widehat{\alpha\b(x_i)}\dots,\alpha\b(x_{n+1}),\alpha^n(x_i)\cdot x_{n+2})\\
\nonumber&&+\sum_{1\leq i<j\leq n+1}(-1)^{i+j} \partial^nf([\b(x_i),\alpha(x_j)]_C,\alpha\b(x_1),\dots,\widehat{\alpha\b(x_i)},\dots,\widehat{\alpha\b(x_j)},\dots,\alpha\b(x_{n+1}),\b(x_{n+2}))\\
\nonumber&&=\Omega_1+\Omega_2-\Omega_3+\Omega_4.
\end{eqnarray}
Where
{\small\begin{align}
 &\Omega_1
 \nonumber=\sum_{i=1}^{n+1}(-1)^{i+1}L(\alpha^{n}\b^{n}(x_i))\partial^nf(\alpha(x_1),\dots,\widehat{\alpha(x_i)},\dots,\alpha(x_{n+1}),x_{n+2})\\
 \label{e11}&=\sum_{i=1}^{n+1}\sum_{\stackrel{j=1}{j<i}}^n(-1)^{i+j}L(\alpha^{n}\b^{n}(x_i))L(\alpha^{n}\b^{n-1}(x_j))f(\alpha^2(x_1),..,\widehat{x_{j,i}}
,..,\alpha^2(x_{n+1}),x_{n+2})\\
\label{e12} &+\sum_{i=1}^{n+1}\sum_{\stackrel{j=1}{j>i}}^n(-1)^{i+j-1}L(\alpha^{n}\b^{n}(x_i))L(\alpha^{n}\b^{n-1}(x_j))f(\alpha^2(x_1),..,\widehat{x_{i,j}},..,\alpha^2(x_{n+1}),x_{n+2})\\
\label{e13} &+\sum_{i=1}^{n+1}\sum_{\stackrel{j=1}{j<i}}^n(-1)^{i+j}L(\alpha^{n}\b^{n}(x_i))R(\b^{n-1}(x_{n+2}))f(\alpha\b(x_1),..,\widehat{x_{j,i}},..,\alpha\b(x_{n+1}),
\alpha^n(x_j))\\
\label{e14} &+\sum_{i=1}^{n+1}\sum_{\stackrel{j=1}{j>i}}^n(-1)^{i+j-1}L(\alpha^{n}\b^{n}(x_i))R(\b^{n-1}(x_{n+2}))f(\alpha\b(x_1),..,\widehat{x_{i,j}},..,\alpha\b(x_{n+1}),\alpha^n(x_{j}))\\
\label{e15} &-\sum_{i=1}^{n+1}\sum_{\stackrel{j=1}{j<i}}^n(-1)^{i+j}L(\alpha^{n}\b^{n}(x_i))f(\alpha^2\b(x_1),..,\widehat{x_{j,i}},..,\alpha^2\b(x_{n+1}),\alpha^n(x_j)\cdot x_{n+2})\\
\label{e16} &-\sum_{i=1}^{n+1}\sum_{\stackrel{j=1}{j>i}}^n(-1)^{i+j-1}L(\alpha^{n}\b^{n}(x_i))f(\alpha^2\b(x_1),..,\widehat{x_{i,j}},..,\alpha^2\b(x_{n+1}),\alpha^n(x_j)\cdot x_{n+2})\\
\label{e17} &-\sum_{i=1}^{n+1}\sum_{{j<k<i}}(-1)^{i+j+k}L(\alpha^{n}\b^{n}(x_i))f([\alpha\b(x_j),\alpha^2(x_k)]_C,\alpha^2\b(x_1),..,\widehat{x_{j,k,i}},..,\alpha^2\b(x_{n+1}),\b(x_{n+2}))\\
\label{e18} &+\sum_{i=1}^{n+1}\sum_{{j<i<k}}(-1)^{i+j+k}L(\alpha^{n}\b^{n}(x_i))f([\alpha\b(x_j),\alpha^2(x_k)]_C,\alpha^2\b(x_1),..,\widehat{x_{j,i,k}},..,\alpha^2\b(x_{n+1}),\b(x_{n+2}))\\
\label{e19} &-\sum_{i=1}^{n+1}\sum_{{i<j<k}}(-1)^{i+j+k}L(\alpha^{n}\b^{n}(x_i))f([\alpha\b(x_j),\alpha^2(x_k)]_C,\alpha^2\b(x_1),..,\widehat{x_{i,j,k}},..,\alpha^2\b(x_{n+1}),\b(x_{n+2}))
\end{align}}
and
{\small\begin{align}
& \Omega_2\nonumber=\sum_{i=1}^{n+1}(-1)^{i+1}R(\b^n(x_{n+2}))\partial^nf(\b(x_1),..,\widehat{\b(x_i)},..,\b(x_{n+1}),\alpha^n(x_i))\\
 \label{e21}&=\sum_{i=1}^{n+1}\sum_{\stackrel{j=1}{j<i}}^n(-1)^{i+j}R(\b^n(x_{n+2}))L(\alpha^{n-1}\b^{n}(x_j))
 f(\a\b(x_1),..,\widehat{x_{j,i}},..,\a\b(x_{n+1}),\alpha^n(x_i))\\
 \label{e22}&+\sum_{i=1}^{n+1}\sum_{\stackrel{j=1}{j>i}}^n(-1)^{i+j-1}R(\b^n(x_{n+2}))L(\alpha^{n-1}\a^{n-1}\b^{n}(x_j))
 f(\a\b(x_1),..,\widehat{x_{i,j}},..,\a\b(x_{n+1}),\alpha^n(x_i))\\
 \label{e23}&+\sum_{i=1}^{n+1}\sum_{\stackrel{j=1}{j<i}}^n(-1)^{i+j}R(\b^n(x_{n+2}))R(\alpha^n\b^{n-1}(x_i))
 f(\b^2(x_1),..,\widehat{x_{j,i}},..,\a^{n-1}\b(x_j))\\
 \label{e24}&+\sum_{i=1}^{n+1}\sum_{\stackrel{j=1}{j>i}}^n(-1)^{i+j-1}R(\b^n(x_{n+2}))R(\alpha^n\b^{n-1}(x_i))
 f(\b^2(x_1),..,\widehat{x_{i,j}},..,\a^{n-1}\b(x_j))\\
  \label{e25}&-\sum_{i=1}^{n+1}\sum_{\stackrel{j=1}{j<i}}^n(-1)^{i+j}R(\b^n(x_{n+2}))
 f(\a\b^2(x_1),..,\widehat{x_{j,i}},..,\a\b^2(x_{n+1}),\a^{n-1}\b(x_j)\cdot\a^n(x_i))\\
\label{e26}&-\sum_{i=1}^{n+1}\sum_{\stackrel{j=1}{j>i}}^n(-1)^{i+j-1}R(\b^n(x_{n+2}))
 f(\a\b^2(x_1),..,\widehat{x_{i,j}},..,\a\b^2(x_{n+1}),\a^{n-1}\b(x_j)\cdot\a^n(x_i))\\
\label{e27} &-\sum_{i=1}^{n+1}\sum_{{j<k<i}}(-1)^{i+j+k}R(\b^n(x_{n+2}))f([\b^2(x_j),\alpha^2(x_k)]_C,\alpha\b^2(x_1),..,\widehat{x_{j,k,i}},..,\alpha\b^2(x_{n+1}),\a^n\b(x_{i}))\\
\label{e28} &+\sum_{i=1}^{n+1}\sum_{{j<i<k}}(-1)^{i+j+k}R(\b^n(x_{n+2}))f([\b^2(x_j),\alpha^2(x_k)]_C,\alpha\b^2(x_1),..,\widehat{x_{j,i,k}},..,\alpha\b^2(x_{n+1}),\a^n\b(x_{i}))\\
\label{e29} &-\sum_{i=1}^{n+1}\sum_{{i<j<k}}(-1)^{i+j+k}R(\b^n(x_{n+2}))f([\b^2(x_j),\alpha^2(x_k)]_C,\alpha\b^2(x_1),..,\widehat{x_{i,j,k}},..,\alpha\b^2(x_{n+1}),\a^n\b(x_{i})),
\end{align}}

and
{\small\begin{align}
\nonumber&\Omega_3=\sum_{i=1}^{n+1}(-1)^{i+1}  \partial^nf(\alpha\b(x_1),..,\widehat{\alpha\b(x_i)}..,\alpha\b(x_{n+1}),\alpha^n(x_i)\cdot x_{n+2})\\
\label{e31}&=\sum_{i=1}^{n+1}\sum_{\stackrel{j=1}{j<i}}^n(-1)^{i+j}L(\a^n\b^n(x_{j})) f(\a^2\b(x_1),..,\widehat{x_{j,i}},..,\a^2\b(x_{n+1}),\alpha^n(x_i)\cdot x_{n+2})\\
  \label{e32}&+\sum_{i=1}^{n+1}\sum_{\stackrel{j=1}{j>i}}^n(-1)^{i+j-1}L(\a^n\b^n(x_{j})) f(\a^2\b(x_1),..,\widehat{x_{i,j}},..,\a^2\b(x_{n+1}),\alpha^n(x_i)\cdot x_{n+2})\\
 \label{e33}&+\sum_{i=1}^{n+1}\sum_{\stackrel{j=1}{j<i}}^n(-1)^{i+j}R(\alpha^n\b^{n-1}(x_i)\cdot\b^{n-1}(x_{n+2}))
 f(\a\b^2(x_1),..,\widehat{x_{j,i}},..,\a^2\b(x_{n+1}),\a^n\b(x_j))\\
 \label{e34}&+\sum_{i=1}^{n+1}\sum_{\stackrel{j=1}{j>i}}^n(-1)^{i+j}R(\alpha^n\b^{n-1}(x_i)\cdot\b^{n-1}(x_{n+2}))
 f(\a\b^2(x_1),..,\widehat{x_{i,j}},..,\a^2\b(x_{n+1}),\a^n\b(x_j))\\
\label{e35}&-\sum_{i=1}^{n+1}\sum_{\stackrel{j=1}{j<i}}^n(-1)^{i+j} f(\a^2\b^2(x_1),..,\widehat{x_{j,i}},..,\a^2\b^2(x_{n+1}),(\a^n\b(x_{j}))\cdot(\alpha^n(x_i)\cdot x_{n+2}))\\
 \label{e36}&-\sum_{i=1}^{n+1}\sum_{\stackrel{j=1}{j>i}}^n(-1)^{i+j} f(\a^2\b^2(x_1),..,\widehat{x_{i,j}},..,\a^2\b^2(x_{n+1}),(\a^n\b(x_{j}))\cdot(\alpha^n(x_i)\cdot x_{n+2}))\\
 \label{e37} &-\sum_{i=1}^{n+1}\sum_{{j<k<i}}(-1)^{i+j+k}f([\alpha\b^2(x_j),\alpha^2\beta(x_k)]_C,\alpha^2\b^2(x_1),..,\widehat{x_{j,k,i}},..,\alpha^2\b^2(x_{n+1}),\a^n\b(x_i)\cdot \b(x_{n+2}))\\
\label{e38} &+\sum_{i=1}^{n+1}\sum_{{j<i<k}}(-1)^{i+j+k}f([\alpha\b^2(x_j),\alpha^2\beta(x_k)]_C,\alpha^2\b^2(x_1),..,\widehat{x_{j,i,k}},..,\alpha^2\b^2(x_{n+1}),\a^n\b(x_i)\cdot\b(x_{n+2}))\\
\label{e39} &-\sum_{i=1}^{n+1}\sum_{{i<j<k}}(-1)^{i+j+k}f([\alpha\b^2(x_j),\alpha^2\beta(x_k)]_C,\alpha^2\b^2(x_1),..,\widehat{x_{i,j,k}},..,\alpha^2\b^2(x_{n+1}),\a^n\b(x_i)\cdot \b(x_{n+2}))
\end{align}}
and
{\small\begin{align}
\nonumber&\Omega_4= \sum_{1\leq i<j\leq n+1}(-1)^{i+j} \partial^nf([\b(x_i),\alpha(x_j)]_C,\alpha\b(x_1),..,\widehat{x_{i,j}},..,\alpha\b(x_{n+1}),\b(x_{n+2}))\\
\label{e41}&=\sum_{\stackrel{1\leq i\leq j\leq n+1}{}}(-1)^{i+j}L([\a^{n-1}\b^n(x_i),\a^n\b^{n-1}(x_j)]_C) f(\a^2\b(x_1),..,\widehat{x_{i,j}},..,\a^2\b(x_{n+1}),\b(x_{n+2}))\\
\label{e42}&+\sum_{\stackrel{1\leq i\leq j}{}}^{ n+1}\sum_{{k<i}}(-1)^{i+j+k+1}L(\a^n\b^n(x_k)) f([\a\b(x_i),\a^2(x_j)]_C ,\a^2\b(x_1),..,\widehat{x_{k,i,j}},..\a^2\b(x_{n+1}),\b(x_{n+2}))\\
\label{e43}&+\sum_{\stackrel{1\leq i\leq j}{}}^{ n+1}\sum_{{i<k<j}}(-1)^{i+j+k}L(\a^n\b^n(x_k))f([\a\b(x_i),\a^2(x_j)]_C ,\a^2\b(x_1),..,\widehat{x_{i,k,j}},..\a^2\b(x_{n+1}),\b(x_{n+2}))\\
\label{e44}&+\sum_{\stackrel{1\leq i\leq j}{}}^{ n+1}\sum_{{j<k}}(-1)^{i+j+k+1}L(\a^n\b^n(x_k)) f([\a\b(x_i),\a^2(x_j)]_C ,\a^2\b(x_1),..,\widehat{x_{i,j,k}},..\a^2\b(x_{n+1}),\b(x_{n+2}))\\
\label{e45}&+\sum_{\stackrel{1\leq i\leq j}{}}^{ n+1}(-1)^{i+j} R(\b^n(x_{n+2})) f(\a\b^2(x_1),..,\widehat{x_{i,j}},..,\a\b^2(x_{n+1}),[\a^{n-1}\b(x_i),\a^n(x_j)]_C)\\
\label{e46}&-\sum_{\stackrel{1\leq i\leq j}{}}^{n+1}\sum_{{k<i}}(-1)^{i+j+k} R(\b^n(x_{n+2})) f([\b^2(x_i),\a\b(x_j)]_C,\a\b^2(x_1),..,\widehat{x_{k,i,j}},..,\a\b^2(x_{n+1}),\a^n\b(x_k))\\
\label{e47}&+\sum_{\stackrel{1\leq i\leq j}{}}^{ n+1}\sum_{{i<k<j}}(-1)^{i+j+k} R(\b^n(x_{n+2})) f([\b^2(x_i),\a\b(x_j)]_C,\a\b^2(x_1),..,\widehat{x_{i,k,j}},..,\a\b^2(x_{n+1}),\a^n\b(x_k))\\
\label{e48}&-\sum_{\stackrel{1\leq i\leq j}{}}^{n+1}\sum_{{j<k}}(-1)^{i+j+k} R(\b^n(x_{n+2})) f([\b^2(x_i),\a\b(x_j)]_C,\a\b^2(x_1),..,\widehat{x_{i,j,k}},..,\a\b^2(x_{n+1}),\a^n\b(x_k))
\end{align}
\begin{align}
\label{e51}-&\sum_{\stackrel{1\leq i\leq j}{}}^{n+1}(-1)^{i+j}  f(\a^2\b^2(x_1),..,\widehat{x_{i,j}},..,\a^2\b^2(x_{n+1}),[\a^{n-1}\b(x_i),\a^n(x_j)]_C\cdot \b^n(x_{n+2}))\\
\label{e52}+&\sum_{\stackrel{1\leq i\leq j}{}}^{n+1}\sum_{{k<i}}(-1)^{i+j+k}  f([\a\b^2(x_i),\a^2\b(x_j)],\a^2\b^2(x_1),..,\widehat{x_{k,i,j}},..,\a^2\b^2(x_{n+1}),\a^n\b(x_k)\cdot \b(x_{n+2}))\\
\label{e53}-&\sum_{\stackrel{1\leq i\leq j}{}}^{n+1}\sum_{{i<k<j}}(-1)^{i+j+k} f([\a\b^2(x_i),\a^2\b(x_j)],\a^2\b^2(x_1),..,\widehat{x_{i,k,j}},..,\a^2\b^2(x_{n+1}),\a^n\b(x_k)\cdot \b(x_{n+2}))\\
\label{e54}+&\sum_{\stackrel{1\leq i\leq j}{}}^{n+1}\sum_{{j<k}}(-1)^{i+j+k}f([\a\b^2(x_i),\a^2\b(x_j)],\a^2\b^2(x_1),..,\widehat{x_{i,j,k}},..,\a^2\b^2(x_{n+1}),\a^n\b(x_k)\cdot \b(x_{n+2}))\\
\label{e55}-&\sum_{\stackrel{1\leq i\leq j}{}}^{n+1}\sum_{{l<i}}(-1)^{i+j+l}  f([[\b^2(x_i),\a\b(x_j)],\a^2\b(x_l)],\a^2\b^2(x_1),..,\widehat{x_{l,i,j}},..,
\a^2\b^2(x_{n+1}),\b^2(x_{n+2}))\\
\label{e56}+&\sum_{\stackrel{1\leq i\leq j}{}}^{n+1}\sum_{{i<l<j}}(-1)^{i+j+l}  f([[\b^2(x_i),\a\b(x_j)],\a^2\b(x_l)],\a^2\b^2(x_1),..,\widehat{x_{i,l,j}},..,
\a^2\b^2(x_{n+1}),\b^2(x_{n+2}))\\
\label{e57}-&\sum_{\stackrel{1\leq i\leq j}{}}^{n+1}\sum_{{j<l}}(-1)^{i+j+l}  f([[\b^2(x_i),\a\b(x_j)],\a^2\b(x_l)],\a^2\b^2(x_1),..,\widehat{x_{i,j,l}},..,
\a^2\b^2(x_{n+1}),\b^2(x_{n+2}))\\
\nonumber&+\sum_{\stackrel{1\leq i\leq j}{}}^{n+1}\sum_{{k<l<i}}(-1)^{i+j+l+k}  f([\a\b^2(x_k),\a^2\b(x_l)]_C,[\a\b^2(x_i),\a^2(x_j)]_C,\a^2\b^2(x_1),..,\\
\label{e58}&\widehat{x_{k,l,i,j}},..,
\a^2\b^2(x_{n+1}),\b^2(x_{n+2}))\\
\nonumber&-\sum_{\stackrel{1\leq i\leq j}{}}^{n+1}\sum_{{k<i<l<j}}(-1)^{i+j+l+k}  f([\a\b^2(x_k),\a^2\b(x_l)]_C,[\a\b^2(x_i),\a^2(x_j)]_C,\a^2\b^2(x_1),..,\\
\label{e59}&\widehat{x_{k,i,l,j}},..,
\a^2\b^2(x_{n+1}),\b^2(x_{n+2}))\\
\nonumber&+\sum_{\stackrel{1\leq i\leq j}{}}^{n+1}\sum_{{k<i<j<l}}(-1)^{i+j+l+k}  f([\a\b^2(x_k),\a^2\b(x_l)]_C,[\a\b^2(x_i),\a^2(x_j)]_C,\a^2\b^2(x_1),..,\\
\label{e510}&\widehat{x_{k,i,j,l}},..,\a^2\b^2(x_{n+1}),\b^2(x_{n+2}))\\
\nonumber&-\sum_{\stackrel{1\leq i\leq j}{}}^{n+1}\sum_{{i<k<l<j}}(-1)^{i+j+l+k}  f([\a\b^2(x_k),\a^2\b(x_l)]_C,[\a\b^2(x_i),\a^2(x_j)]_C,\a^2\b^2(x_1),..,\\
\label{e511}&\widehat{x_{i,k,l,j}},..,
\a^2\b^2(x_{n+1}),\b^2(x_{n+2}))\\
\nonumber&+\sum_{\stackrel{1\leq i\leq j}{}}^{n+1}\sum_{{i<k<j<l}}(-1)^{i+j+l+k}  f([\a\b^2(x_k),\a^2\b(x_l)]_C,[\a\b^2(x_i),\a^2(x_j)]_C,\a^2\b^2(x_1),..,\\
\label{e512}&\widehat{x_{i,k,j,l}},..,
\a^2\b^2(x_{n+1}),\b^2(x_{n+2}))\\
\nonumber&-\sum_{\stackrel{1\leq i\leq j}{}}^{n+1}\sum_{{j<k<l}}(-1)^{i+j+l+k}  f([\a\b^2(x_k),\a^2\b(x_l)]_C,[\a\b^2(x_i),\a^2(x_j)]_C,\a^2\b^2(x_1),..,\\
\label{e513}&\widehat{x_{i,j,k,l}},..,
\a^2\b^2(x_{n+1}),\b^2(x_{n+2})).
\end{align}}
In $\Omega_1,... \Omega_4$.  $\widehat{x_{i, j, k, l}}$ means that we omits the items $x_i, x_j, x_k, x_l$.

Using Eq.\eqref{rep2}, we have
$\eqref{e11}+\eqref{e12}+\eqref{e41}=0.$
Using Eq.\eqref{rep3}, we have
$$\eqref{e13}+\eqref{e21}+\eqref{e23}-\eqref{e33}=\eqref{e14}+\eqref{e22}+\eqref{e24}-\eqref{e34}=0.$$
 It evident to see that
$$\eqref{e19}+\eqref{e42}=\eqref{e18}+\eqref{e43}=\eqref{e17}+\eqref{e44}=\eqref{e15}-\eqref{e32}=\eqref{e16}-\eqref{e31}=\eqref{e25}+\eqref{e26}+\eqref{e45}=0,$$
and
$$ \eqref{e27}+\eqref{e48}=\eqref{e28}+\eqref{e47}=\eqref{e29}+\eqref{e46}=\eqref{e37}-\eqref{e54}=\eqref{e38}-\eqref{e53}=\eqref{e39}-\eqref{e52}=0.$$

Using identity of BiHom-left-symmetric algebra Eq.$\eqref{id-Bihom-pre}$, we have
$$
\eqref{e35}+\eqref{e36}-\eqref{e51}=0.
$$
Since $A^C$ is BiHom-Lie algebra, then using the BiHom-Jacobi identity \eqref{Bihom-jaco} we have $\eqref{e55}+\eqref{e56}+\eqref{e57}=0$. Finally, it easy to see that $\eqref{e58}+\dots+\eqref{e513}=0$.
\end{proof}

Denote the set of  $n$-cocycles by $Z^n(A;V)=\ker(\partial^n)$ and the set of  $n$-cobords by $B^n(A;V)=Im(\partial^{n-1})$. We denote by $H^n(A;V)=Z^n(A;V)/B^n(A;V)$ the corresponding cohomology groups of the BiHom-left-symmetric algebra $(A,\cdot,\alpha,\b)$ with the coefficient in the representation $(V,L,R,\f,\p)$.

\section{Linear deformations of BiHom-left-symmetric algebras}
In this section, we study linear deformations of BiHom-left-symmetric algebras using the cohomology defined in the previous section, and introduce the notion of a Nijenhuis operator on a BiHom-left-symmetric algebra. We show that a Nijenhuis operator gives rise to a trivial deformation.

Let $(A,\cdot,\alpha,\b)$ be a BiHom-left-symmetric algebra. In the sequel, we will also denote the BiHom-left-symmetric multiplication $\cdot$ by $\Pi$.

\begin{defi}
Let $(A,\Pi,\a,\b)$ be a BiHom-left-symmetric algebra and $\pi \in C^2(A;A)$ be a bilinear operator, satisfying $\pi\circ (\alpha\otimes\alpha)=\alpha\circ \pi$ and $\pi \circ (\b\otimes\b)=\b\circ \pi$. If  $$\Pi_t=\Pi+t\pi$$ is still a BiHom-left-symmetric multiplication on $A$ for all $t$, we say that $\omega$ generates a (one-parameter) linear deformation of a BiHom-left-symmetric algebra $(A,\Pi,\a,\b)$.
\end{defi}
It is direct to check that $\Pi_t=\Pi+t\pi$ is a linear deformation of a Hom-left-symmetric algebra $(A,\Pi,\alpha,\b)$ if and only if
 for any $x,y,z\in \g$, the following two equalities are satisfied:
\begin{align}
\nonumber& \pi(\b(x),\a(y))\cdot \b(z)+\pi(\b(x)\cdot \a(y),\b(z))-\a\b(x)\cdot \pi(\a(y),z)-\pi(\alpha\b(x),\a(y)\cdot z) \\
\label{eq:4.1}  -&\pi(\b(y),\a(x))\cdot \b(z)-\pi(\b(y)\cdot \a(x),\b(z))+\alpha\b(y)\cdot \pi(\a(x),z)+\pi(\alpha\b(y),\a(x)\cdot z)=0,\\
\label{eq:4.2}& \pi(\pi(\b(x),\a(y)),\b(z))-\pi(\alpha\b(x),\pi(\a(y),z))-
\pi(\pi(\b(y),\a(x)),\b(z))+\pi(\alpha\b(y),\pi(\a(x),z))=0.
\end{align}
Obviously, \eqref{eq:4.1} means that $\pi$ is a $2$-cocycle of the BiHom-left-symmetric algebra $(A,\cdot,\alpha,\b)$ with the coefficient in the adjoint representation, i.e. $$\partial^2\pi(x,y,z)=0.$$ Furthermore, \eqref{eq:4.2} means that $(A,\pi,\alpha,\b)$ is  a BiHom-left-symmetric algebra.

\begin{defi}
Let $(A,\Pi,\a,\b)$ be a BiHom-left-symmetric algebra. Two linear deformations $\Pi_t^1=\Pi+t\pi_1$ and
$\Pi_t^2=\Pi+t\pi_2$ are said to be  equivalent if there exists a
linear operator $N\in \gl(A)$ such that $T_t={\Id}+tN$ is a
BiHom-left-symmetric algebra homomorphism from $(A,\Pi_t^2,\alpha,\b)$ to $(A,\Pi_t^1,\alpha,\b)$. In particular, a
linear deformation $\Pi_t=\Pi+t\pi$ of a BiHom-left-symmetric algebra $(A,\Pi,\alpha,\b)$ is said to be  trivial if there exists a linear operator $N\in \gl(A)$ such that $T_t={\Id}+tN$ is a
BiHom-left-symmetric algebra homomorphism from  $(A,\Pi_t,\alpha,\b)$ to  $(A,\Pi,\alpha,\b)$.
\end{defi}

Let $T_t={\Id}+tN$ be a homomorphism from
$(A,\Pi_t^2,\alpha,\b)$ to $(A,\Pi_t^1,\alpha,\b)$. By \eqref{bihomo-2}-\eqref{bihomo-3}, we have
\begin{equation}\label{equi-deformation-0}
N\circ\alpha=\alpha\circ N\ \textrm{and}\  N \circ \b= \b \circ N.
\end{equation}
Then by \eqref{bihomo-1} we obtain
\begin{eqnarray}
\label{equi-deformation-1}&\pi_2(x,y)-\pi_1(x,y)=N(x)\cdot y+x\cdot N(y)-N(x\cdot y),&\\
\label{equi-deformation-2}&\pi_1(x,N(y))+\pi_1(N(x),y)=N(\pi_2(x,y))-N(x)\cdot N(y),&\\
\label{equi-deformation-3}&\pi_1(N(x),N(y))=0.&
\end{eqnarray}
Note that \eqref{equi-deformation-1} means that $\pi_2-\pi_1=\partial^1(N)$. Thus, we have

\begin{thm}\label{thm:iso3} Let $(A,\Pi,\alpha,\b)$ be a BiHom-left-symmetric algebra.
  If two linear deformations $\Pi_t^1=\Pi+t\pi_1$ and
$\Pi_t^2=\Pi+t\pi_2$ are equivalent, then $\pi_1$ and $\pi_2$ are in the same cohomology class of $H^2(A;A)$.
\end{thm}

\begin{defi}\label{defi:operator}
Let $(A,\cdot,\alpha,\b)$ be a BiHom-left-symmetric algebra. A linear operator $N\in \gl(A)$ is called a Nijenhuis operator on $(A,\cdot,\alpha,\b)$ if $N$ satisfies \eqref{equi-deformation-0} and the equation
      \begin{eqnarray}
        N(x)\cdot N(y)=N(x\cdot_N y),\quad \forall x,y\in A.
         \label{eq:Nij3}
        \end{eqnarray}
        where the product $\cdot_N$ is defined by
\begin{equation}\label{eq:5.2}
x\cdot_N y:=  N(x)\cdot y+x\cdot N(y)-N(x\cdot y).
\end{equation}
\end{defi}

By \eqref{equi-deformation-0}-\eqref{equi-deformation-3}, a trivial linear deformation of a BiHom-left-symmetric algebra gives rise to a Nijenhuis operator $N$. Conversely, a Nijenhuis operator can also generate a trivial linear deformation as the following theorem shows.

\begin{thm}\label{Nij-deformation}
Let $N$ be a Nijenhuis operator on a BiHom-left-symmetric algebra $(A,\Pi,\alpha,\b)$. Then a linear deformation $(A,\Pi_t,\alpha,\b)$ of the BiHom-left-symmetric algebra $(A,\Pi,\a,\b)$ can be obtained by putting
\begin{equation}\label{trivial-deformation-generate}
\pi(x,y)=\partial^1(N)(x,y)=x\cdot_N y.
\end{equation}
Furthermore, the linear deformation $(A,\Pi_t,\a,\b)$ is trivial.
\end{thm}
\begin{proof}
 Using \eqref{trivial-deformation-generate}, we have $\pi\circ(\alpha\otimes\alpha)=\alpha\circ\pi$, $\pi\circ(\b\otimes\b)=\b\circ\pi$ and $\partial^2\pi=0$. To show that $\pi$ generates a linear deformation of
the BiHom-left-symmetric algebra $(A,\cdot,\a,\b)$, we only need to verify that~\eqref{eq:4.2} holds. This means that, $(A,\pi,\a,\b)$ is  a BiHom-left-symmetric algebra. By simple computations, we obtain
\begin{align*}
&\pi(\pi(\b(x),\a(y)),\b(z))-\pi(\a\b(x),\pi(\a(y),z))-\pi(\pi(\b(y),\a(x)),\b(z))+\pi(\a\b(y),\pi(\a(x),z))\\
=&(N(\b(x))\c N(\a(y))) \c \b(z) +(N(\b(x))\cdot \a(y))\cdot N(\b(z))+(\b(x)\cdot N(\a(y)))\cdot N(\b(z)) \\
-&N(\b(x)\cdot \a(y))\cdot N(\b(z))
-N\big((N(\b(x))\cdot \a(y))\cdot \b(z)\big)-N\big((\b(x)\cdot N(\a(y)))\cdot \b(z)\big) \\
+&N\big(N(\b(x)\cdot \a(y))\cdot \b(z)\big)
-N(\a\b(x))\cdot (N(\a(y))\cdot z)-N(\a\b(x))\cdot (\a(y)\cdot N(z)) \\
+&N(\a\b(x))\cdot N(\a(y)\cdot z)
- \a\b(x)\c (N(\a(y)) \c N(z))
+N\big(\a\b(x)\cdot (N(\a(y))\cdot z)\big) \\
+&N\big(\a\b(x)\cdot (\a(y)\cdot N(z))\big)-N\big(\a\b(x)\cdot N(\a(y)\cdot z)\big)
- (N(\b(y)) \c N(\a(x))) \c \b(z) \\
-&(N(\b(y))\cdot \a(x))\cdot N(\b(z))-(\b(y)\cdot N(\a(x)))\cdot N(\b(z))+N(\b(y)\cdot \a(x))\cdot N(\b(z))\\
+&N\big((N(\b(y))\cdot \a(x))\cdot \b(z)\big)+N\big((\b(y)\cdot N(\a(x)))\cdot \b(z)\big)-N\big(N(\b(y)\cdot \a(x))\cdot \b(z)\big)\\
+&N(\a\b(y))\cdot (N(\a(x))\cdot z)+N(\a\b(y))\cdot (\a(x)\cdot N(z))-N(\a\b(y))\cdot N(\a(x)\cdot z)\\
+& \a\b(y) \c (N(\a(x)) \c N(z))
-N\big(\a\b(y)\cdot (N(\a(x))\cdot z)\big)\\
-&N\big(\a\b(y)\cdot (\a(x)\cdot N(z))\big)+N\big(\a\b(y)\cdot N(\a(x)\cdot z)\big).
\end{align*}

Since $N$ commutes with $\a$ and $\b$, by the definition of a BiHom-left-symmetric algebra, we get
\begin{align*}
&(N(\b(x))\cdot \a(y))\cdot N(\b(z))-N(\a\b(x))\cdot (\a(y)\cdot N(z))\\
 -&(\b(y)\cdot N(\a(x)))\cdot N(\b(z))+\a\b(y)\cdot (N(\a(x))\cdot N(z))=0,\\
 &(N(\b(y))\cdot \a(x))\cdot N(\b(z))-N(\a\b(y))\cdot (\a(x)\cdot N(z))\\
 -&(\b(x)\cdot N(\a(y)))\cdot N(\b(z))+\a\b(x)\cdot (N(\a(y))\cdot N(z))=0,\\
 &(N(\b(x))\cdot N(\a(y)))\cdot \b(z)-N(\a\b(x))\cdot (N(\a(y))\cdot z)\\
-&(N(\b(y))\cdot N(\a(x)))\cdot \b(z)+N(\a\b(y))\cdot (N(\a(x))\cdot z)=0.
\end{align*}

Since $N$ is a Nijenhuis operator, by the definition of a BiHom-left-symmetric algebra, we obtain
{\small
\begin{align*}
-&N(\b(x)\cdot \a(y))\cdot N(\b(z))+N(\a\b(x))\cdot N(\a(y)\cdot z)+N(\b(y)\cdot \a(x))\cdot N(\b(z))-N(\a\b(y))\cdot N(\a(x)\cdot z)\\
=&-N\big((\b(x)\cdot \a(y))\cdot_N \b(z)\big)+N\big(\a\b(x)\cdot_N (\a(y)\cdot z)\big)
+N\big((\b(y)\cdot \a(x))\cdot_N \b(z)\big)-N\big(\a\b(y)\cdot_N (\a(x)\cdot z)\big)\\
=&-N\big(N(\b(x)\cdot \a(y))\cdot \b(z)\big)-N\big((\b(x)\cdot \a(y))\cdot N(\b(z))\big)
+N\big(N(\a\b(x))\cdot (\a(y)\cdot z)\big)+N\big(\a\b(x)\cdot N(\a(y)\cdot z)\big)\\
+&N\big(N(\b(y)\cdot \a(x))\cdot \b(z)\big)+N\big((\b(y)\cdot \a(x))\cdot N(\b(z))\big)
-N\big(N(\a\b(y))\cdot (\a(x)\cdot z)\big)-N\big(\a\b(y)\cdot N(\a(x)\cdot z)\big).
\end{align*}}
Therefore, we have
{\small
\begin{align*}
&\pi(\pi(\b(x),\a(y)),\b(z))-\pi(\a\b(x),\pi(\a(y),z))-\pi(\pi(\b(y),\a(x)),\b(z))+\pi(\a\b(y),\pi(\a(x),z))\\
=&-N\big((\b(x)\cdot \a(y))\cdot N(\b(z))\big)+N\big(\a\b(x)\cdot (\a(y)\cdot N(z))\big)
+N\big((\b(y)\cdot \a(x))\cdot N(\b(z))\big) \\
-&N\big(\a\b(y)\cdot (\a(x)\cdot N(z))\big)
-N\big((N(\b(x))\cdot \a(y))\cdot \b(z)\big)+N\big(N(\a\b(x))\cdot (\a(y)\cdot z)\big) \\
+&N\big((\b(y)\cdot N(\a(x)))\cdot \b(z)\big)-N\big(\a\b(y)\cdot (N(\a(x))\cdot z)\big)
+N\big((N(\b(y))\cdot \a(x))\cdot \b(z)\big) \\
-&N\big(N(\a\b(y))\cdot (\a(x)\cdot z)\big)
-N\big((\b(x)\cdot N(\a(y)))\cdot \b(z)\big)+N\big(\a\b(x)\cdot (N(\a(y))\cdot z)\big)
=0.
\end{align*}
}
Thus, $\pi$ generates a linear deformation of the BiHom-left-symmetric algebra $(A,\cdot,\a,\b)$.

We define $T_t={\Id}+tN$. It is straightforward to deduce that $T_t={\Id}+tN$ is a BiHom-left-symmetric algebra homomorphism from $(A,\Pi_t,\a,\b)$ to  $(A,\Pi,\a,\b)$. Thus, the deformation generated by $\pi=\partial^1(N)$ is trivial.
\end{proof}

By \eqref{eq:Nij3} and Theorem \ref{Nij-deformation}, we have the following corollary.
\begin{cor}
Let $N$ be a Nijenhuis operator on a BiHom-left-symmetric algebra $(A,\pi,\a,\b)$, then $(A,\cdot_N,\a,\b)$ is a BiHom-left-symmetric algebra, and $N$ is a homomorphism from $(A,\cdot_N,\a,\b)$ to $(A,\cdot,\a,\b)$.
\end{cor}
Now, we recall linear deformations of BiHom-Lie algebras and Nijenhuis operators on BiHom-Lie algebras, which give trivial deformations of BiHom-Lie algebras.
\begin{defi}\label{defi:deformation}
Let $(\g,[\cdot,\cdot],\a,\b)$ be a BiHom-Lie algebra and $\pi \in \huaC^2(\g;\g)$ a skew-symmetric bilinear operator, satisfying $\pi\circ (\a\otimes\a)=\a\circ \pi$ and $\pi\circ (\b\otimes\b)=\b\circ \pi$. Consider a $t$-parameterized family of bilinear operations, $$[\cdot,\cdot]_t=[\cdot,\cdot]+t\pi.$$ If $(\g,[\cdot,\cdot]_t,\a,\b)$ is a BiHom-Lie algebra for all $t$, we say that $\pi$ generates a (one-parameter) linear deformation of a BiHom-Lie algebra $(\g,[\cdot,\cdot],\a,\b)$.
    \end{defi}
\begin{defi}
Let $(\g,[\cdot,\cdot],\a,\b)$ be a BiHom-Lie algebra. A linear operator $N\in \gl(A)$ is called a Nijenhuis operator on $(\g,[\cdot,\cdot],\a,\b)$ if we have
\begin{eqnarray}
\label{homo-4}N\circ \alpha&=&\alpha\circ N,\\
\label{homo-5}[N(x),N(y)]&=&N[x,y]_N,\quad \forall x,y\in \g.
\end{eqnarray}
where the bracket $[\cdot,\cdot]_N$ is defined by
\begin{equation}
[x,y]_N\triangleq  [N(x),y]+[x,N(y)]-N[x,y].
\end{equation}
    \end{defi}
\begin{pro}
If $\pi \in C^2(A;A)$ generates a linear deformation of a BiHom-left-symmetric algebra $(A,\cdot,\a,\b)$, then $\pi_C \in \huaC^2(A^C;A)$ defined by$$\pi_C(x,y)=\pi(x,y)-\pi(\a^{-1}\b(y),\a\b^{-1}(x))$$ generates a linear deformation of the sub-adjacent BiHom-Lie algebra $A^C$.
\end{pro}

\begin{proof}
Assume that $\pi$ generates a linear deformation of a BiHom-left-symmetric algebra $(A,\cdot,\a,\b)$. Then $(A,\Pi_t,\a,\b)$ is a BiHom-left-symmetric algebra. Consider its corresponding sub-adjacent BiHom-Lie algebra $(A,[\cdot,\cdot]_t,\a,\b)$, we have
\begin{eqnarray*}
  [x,y]_t&=&\P_t(x,y)-\P_t(\a^{-1}\b(y),\a\b^{-1}(x))\\
  &=&x\cdot y+t\pi(x,y)-\a^{-1}\b(y)\cdot \a\b^{-1}(x)-t\pi(\a^{-1}\b(y),\a\b^{-1}(x))\\
  &=&[x,y]_C+t\p_C(x,y).
\end{eqnarray*}
Thus, $\p_C$ generates a linear deformation of $A^C$.
\end{proof}

\begin{pro}
If $N$ is a Nijenhuis operator on a BiHom-left-symmetric algebra $(A,\cdot,\a,\b)$, then $N$ is a Nijenhuis operator on the subadjacent BiHom-Lie algebra $A^C$.
\end{pro}
\begin{proof}
Since $N$ is a Nijenhuis operator on $(A,\cdot,\a,\b)$, we have $N\circ \a=\a\circ N$ and $N \circ \b=\b \circ N$. For all $x,y\in A$, by Definition \ref{defi:operator}, we have
\begin{align*}
 & [N(x),N(y)]_C=N(x)\cdot N(y)-\a^{-1}\b(N(y))\cdot \a\b^{-1}(N(x))\\
  =&N\big(N(x)\cdot y+x\cdot N(y)-N(x\cdot y)-N(\a^{-1}\b(y))\cdot \a\b^{-1}(x)-\a^{-1}\b(y)\cdot N(\a\b^{-1}(x))+N(\a^{-1}\b(y)\cdot \a\b^{-1}(x))\big)\\
  =&N([N(x),y]_C+[x,N(y)]_C-N[x,y]_C).
\end{align*}
Thus, $N$ is a Nijenhuis operator on $A^C$.
\end{proof}


\end{document}